\documentclass[12pt,article]{amsart}

\usepackage[margin=1in]{geometry}
\usepackage{amsmath, amssymb,amsthm,url,mathrsfs,graphicx,amscd,amsfonts}
\usepackage[mathscr]{eucal}
\usepackage{dsfont}
\usepackage{mathtools}
\usepackage{hyperref,graphicx}

\usepackage[utf8]{inputenc}
\usepackage{graphicx,todonotes,hyperref}
\usepackage{epstopdf}
\usepackage{inputenc}
\usepackage{enumitem}
\usepackage{amsmath}
\usepackage{xcolor}
\usepackage{amssymb}
\newtheorem*{acknowledgements*}{Acknowledgements}
\newtheorem{hyp}{Hypothesis}

\allowdisplaybreaks
\newtheorem{theorem}{Theorem}[section]
\newtheorem{lemma}[theorem]{Lemma}
\newtheorem{corollary}[theorem]{Corollary}
\newtheorem{proposition}[theorem]{Proposition}

\newtheorem{conjecture}{Conjecture}[section]
\newtheorem*{theorem*}{Theorem}
\theoremstyle{remark}\newtheorem*{remark}{Remark}

\newtheorem{case}{Case}
\newtheorem{subcase}{Subcase}[case]
\numberwithin{equation}{section}

\renewcommand{\phi}{\varphi}

\begin{document}

\title[Selberg's Central limit theorem]{Selberg's Central limit theorem for quadratic dirichlet L-functions over function fields}

\author{Pranendu Darbar}
\address{Indian Statistical Institute   \\
Kolkata  \\
West Bengal 700108,  India}
\email[Pranendu Darbar]{darbarpranendu100@gmail.com}

\author{Allysa Lumley}
\address{Centre de Reserches  Math{\'e}matiques \\
Montr{\'e}al, QC, Canada}
\email[Allysa Lumley]{lumley@crm.umontreal.ca}


\keywords{Finite fields, Function fields, Central limit theorem, Hyperelliptic curves}

\begin{abstract}
In this article, we study the logarithm of the central value $L\left(\frac{1}{2}, \chi_D\right)$ in the symplectic family of Dirichlet $L$-functions associated with the hyperelliptic curve of genus $\delta$ over a fixed finite field $\mathbb{F}_q$ in the limit as $\delta\to \infty$. Unconditionally, we show that the distribution of $\log \big|L\left(\frac{1}{2}, \chi_D\right)\big|$ is asymptotically bounded above by the Gaussian distribution of mean $\frac{1}{2}\log \deg(D)$ and variance $\log \deg(D)$. Assuming a mild condition on the distribution of  the low-lying zeros in this family, we obtain the full Gaussian distribution.
\end{abstract}


\maketitle
\section{Introduction}
The study of the analytic properties of $L$-functions is an integral part of modern number theory due to their myriad of connections to other objects of interest. Of particular interest is trying to understand the distribution of values these $L$-functions take at various points within the critical strip, especially values along the central line $\Re(s)=\frac12$. Perhaps the most significant unconditional result in this vein is Selberg's central limit theorem \cite{SelbergZeta}, which states that, for $T$ sufficiently large as $t$ varies in $[T,2T]$, the real and imaginary parts of $\log \zeta(\tfrac12+it)$ are normally distributed with mean $0$ and variance $\frac12\log\log T$. Following this result, a great deal of effort has gone into generalizing this to other $L$-functions. Indeed, Selberg himself proved that for a fixed $t$, the imaginary part of $\log L(\tfrac12+it, \chi)$ becomes normally distributed as $\chi$ varies among Dirichlet characters of a large prime modulus $q$, \cite{SelberDirichlet}. Under some widely believed conditional assumptions, Bombieri and Hejhal \cite{BombieriHejhal} were able to prove similar results for a fairly general class of $L$-functions, also in the $t$-aspect. This work has been made unconditional in at least one case by Wenzhi Luo \cite{LuoHecke}.\\

\noindent
Recently, Radziwi{\l}{\l} and Soundararajan \cite{RadSoundSCL} provided another proof of Selberg's original result that was much shorter and easier to digest. This proof was adapted by Hsu and Wong \cite{HsuWong} to treat the case of $\log |L(\frac12+it,\chi)|$ as $t$ varies in $[T,2T]$. Within this work, they also make explicit a notion of independence between primitive Dirichlet $L$-functions predicted by Selberg \cite{SelbergConj}. That is, they prove that for $\{\chi_i\}_{i=1}^n$ a sequence of distinct primitive Dirichlet characters and $T$ sufficiently large, as $t$ varies in $[T,2T]$ the vector $(\log|L(\frac12+it,\chi_1)|, \ldots, \log|L(\frac12+it,\chi_n)|)$ becomes an $n$-variate normal distribution with mean vector $0_n$ and covariance matrix $\frac12(\log\log T)I_n$.  \\

\noindent
The above examples have something in common, which Katz and Sarnak first illuminated in their groundbreaking work \cite{KatzandSarnak}. From their work we learned that the  central values of $L(\frac12+it,f)$ of an $L$-function belong in a family with symmetry type governed by the classical compact groups. In the above, the families of $L$-functions described correspond to unitary families, $U(N)$. In \cite{KatzandSarnak} it was proposed that although zeros high up on the critical line follow the statistics of the unitary family that the low-lying zeros of specific families may follow the statistics of the orthogonal or symplectic matrix groups instead. This idea was demonstrated by Katz and Sarnak finding zeta-functions over function fields whose zero statistics exhibited such behaviours. Keating and Snaith \cite{KeatingandSnaith} made some ``Selberg-type'' conjectures for the orthogonal and symplectic families based on calculations from random matrix theory. These conjectures appear to be unreachable at the moment since they involve the real part of the logarithm of $L$-functions at the central point $s=\frac12$. At present, even the best methods in analytic number theory cannot guarantee more than a positive proportion of $L$-functions in any given family is non-zero at a single point. Indeed, even if we can show that the central value is non-negative, the logarithm is still highly sensitive to zeros close to the central point (i.e. the low-lying zeros described in \cite{KatzandSarnak}).  So what can be said about these families? \\

\noindent
For $d>0$ and characters of the form $\chi_{8d}$, Hough \cite{HoughSymplectic}, provides an upper bound which matches the conjectures of \cite{KeatingandSnaith}. In this work, he also \textit{conditionally} proves the conjectures of Keating and Snaith:  Let $\frac{D}2\le d\le D$ and 
\[B(d)=\frac1{\sqrt{\log\log D}}\left(\log|L(\tfrac12,\chi_{8d})|-\frac12\log\log D\right),\]
then in the sense of distributions 
\begin{equation}\label{KSproven}
\frac1{|s(D)|}\sum_{d \in S(D)}\hspace{-0.15cm}\Delta_{B(d)}\to N(0,1) \text{ as } D\to\infty, 
\end{equation}
where $s(D)$ is the set of odd square-free integers between $D/2$ and $D$, $\Delta_x$ is the point mass at $x$ and $N(0,1)$ denotes the standard normal distribution with mean $0$ and variance $1$. He also finds similar results for the family of $L$-functions associated to weight $k$ Hecke eigen-cusp forms which has orthogonal symmetry type. It remains to explain the conditions under which this theorem holds. Recall that the trouble with proving the conjectures of \cite{KeatingandSnaith} stems from the lack of control over the low-lying zeros associated with the family of $L$-functions in question. Thus, any imposed conditions should relate to this. Of course, it is enough to assume the Riemann Hypothesis and the well-known density conjecture, due to Iwaniec, Luo and Sarnak \cite{ILSZDconj}. The specific statement associated to the symplectic family in question is given below: 
\begin{conjecture}[Zero Density Conjecture]
Let $\phi(x)$ be a Schwartz class function on $\mathbb{R}$ with Fourier transform having compact support. Define the density 
\[W(\mathrm{Sp})(x)dx=\left(1-\frac{\sin(2\pi x)}{2\pi x}\right)dx,\]
let $\rho=\frac12+i\gamma$, $\gamma\in\mathbb{C}$, represent the non-trivial zeros of $L(s,\chi_{8d})$ and $\Lambda(s,\chi_{8d})$ represent the associated completed $L$-function.  Then 
\[\lim_{D\to\infty}\frac1{|s(D)|}\sum_{d\in s(D)}\sum_{\Lambda(\rho,\chi_{8d})=0} \phi\left(\frac{\gamma\log D}{2\pi}\right)=\int_{-\infty}^{\infty}\phi(x)W(\textrm{Sp})(x)dx.\]
\end{conjecture}
\noindent
However, these conditions are quite strong and are also pretty far out of reach. It is therefore of interest to find a weaker condition which still allows Hough's theorem to hold.
Indeed, he uses the following hypothesis to prove \eqref{KSproven}.
\begin{hyp}[Hypothesis 1.3, \cite{HoughSymplectic}]\label{Houghhyp} Assume $y=y(D)\to\infty$ with $D$. Then for $d$ such that $8d$ is a fundamental discriminant, we have  
\[\mathbb{P}\left(\frac{D}2\le d \le D : \gamma_{\min}(d)< \frac{2\pi}{y\log D}\right)=o(1) \text{ as } D\to \infty,\]
where \[\gamma_{\min}(d)=\displaystyle{\min_{\substack{L(\rho,\chi_{8d})=0\\ \rho=\frac12+\beta+i\gamma}}}|\gamma|.\]
\end{hyp}

\noindent
This hypothesis comes from some probabilistic reasoning which suggests if $s$ is near $\frac12$ and the conductor of $\chi_{8d}$ is $\asymp D$ then typically $\gamma_{\min}(d)\gg\frac1{\log D}$.  \\ 

In this paper, we prove a Selberg-type theorem over function fields, $\mathbb{F}_q(t)$, where $q$ is a fixed odd prime power, thus, going back to Katz and Sarnak's original idea. Let $D\in \mathbb{F}_q[t]$ be a monic square-free polynomial. We define the primitive quadratic character associated to $D$ as the Kronecker symbol, $\left(\frac{D}{\cdot}\right)$, see section \ref{prelimsec} for more details. Let $\mathcal{H}_{n, q}$ denote the set of monic square-free polynomials of degree $n$ over $\mathbb{F}_q$. 
There are three aspect in function fields: (i) $n$ is fixed and $q$ varies which we call ``large finite field'' aspect; (ii) $q$ fixed and $n$ varies which we call ``large degree'' aspect; (iii) both $n$ and $q$ varies which we call the double limit aspect. Our goal is to explore the ``large degree'' aspect. For the geometrical significance of these aspects see the section $2.1.$

\vspace{2mm}
\noindent
In particular, we prove results for $\log |L(\frac12,\chi_D)|$ in the ``large degree'' aspect, as such we suppress the $q$ in the notation of the set, that is  $\mathcal{H}_{n, q}=\mathcal{H}_{n}$. For information in the remaining ranges of $\sigma$, see work of the second author \cite{Lumley1} and \cite{Lumley2}. In this setting, Weil \cite{WEIL} has proven the Riemann Hypothesis, so one might hope to provide an unconditional result of this flavour. However, in some ways, the situation over function fields is a bit more frustrating. \\

\noindent
We are able to prove the following unconditional results which show that if we are very near the $\frac12$-line then $\log|L(\frac12+\sigma_0,\chi_D)|$ is normally distributed with mean $\frac12\log d(D) $ and variance $\log d(D)$, where $d(D)$ is the degree of the polynomial $D$. 
More precisely, we define 
\begin{align}\label{delta}
2\delta=d(D)-1-\lambda,
\end{align}
where 
\begin{align*}
  \lambda= 
 \left\{
 \begin{array}
 [c]{ll}
  1 & \text{if\, $d(D)$ even} \\
 0 & \text{if\, $d(D)$ odd}.
  \end{array}
 \right.
\end{align*}
Then we have the following theorem:
\begin{theorem}\label{Unconditional CLT near to half line}
Let $\delta$ be defined by \eqref{delta} and $\sigma_o=\sigma_o(\delta)$ be a function of $\delta$, tending to zero as $\delta\to \infty$ in such a way that $\sigma_o \delta\to \infty$ but $\frac{\sigma_o \delta}{\sqrt{\log \delta}}\to 0$. For $D\in \mathcal{H}_n$, we consider
\[
A(D)=\frac{1}{\sqrt{\log n}}\left(\log\left| L\left(\frac{1}{2}+\sigma_o, \chi_D\right)\right|-\frac{1}{2}\log n\right).
\]
Then, as $n\to \infty$
\[
\frac{1}{|\mathcal{H}_n|}\sum_{D\in \mathcal{H}_n}\Delta_{A(D)}\to N(0, 1),
\]
where $\Delta_x$ is the point mass at $x$ and $N(0, 1)$ is the standard normal distribution.
\end{theorem}
\begin{remark}
From \eqref{delta}, we can notice that $2\delta=n-1-\lambda$, thus as $n\to\infty$ so does $\delta$. One introduces the parameter $\delta$ as it helps to give a uniform statement about the number of non-trivial zeros associated to $L(s,\chi_D)$ regardless of whether $n=2g+1$ or $n=2g+2$. In both cases, $\delta=g$ which denotes the genus of the hyperelliptic curve generated from $C_D: y^2=D(x)$. Further details on this are given in Section \ref{specsec}.  
\end{remark}
\begin{remark}
In Section \ref{logformula}, we observe that the main term of our expression for $\log L(\frac12+\sigma,\chi_D)$ is real and the contribution of the imaginary part is very small, thus we we do not consider the distribution of $\arg L(\frac12+\sigma,\chi_D)$.
 \end{remark}

\noindent
Theorem \ref{Unconditional CLT near to half line}  leads us to the following  corollary, which is an unconditional proof of ``one-half'' of the Keating-Snaith conjectures. That is, we show that $\log |L(\frac12,\chi_D)|$ has at most a Gaussian distribution. 
\begin{corollary}\label{Unconditional CLT}
Let $Z$ be a real number. As $n\to \infty$, we have
\begin{align*}
\mathbb{P}\left(D\in \mathcal{H}_n : \frac{1}{\sqrt{\log n}}\left(\log \left|L\left(\frac{1}{2}, \chi_D\right)\right|-\frac{1}{2}\log n\right)>Z \right)\leq \int_Z^{\infty}e^{-\frac{t^2}{2}}dt+o_Z(1).
\end{align*}
\end{corollary}

\vspace{2mm}
\noindent
Of course, we would like to have a full result about the distribution for $\log |L(\frac12,\chi_D)|$ as $D$ varies over $\mathcal{H}_n$, we can at best provide a conditional result, and this comes from some uncertainty due to a conjecture of Chowla.  Chowla's conjecture presented in \cite{ChowlaConj} asserts that {\it all} $L$-functions associated with quadratic characters do not vanish at the central point $s=\frac12$. The analogue of Chowla's conjecture over function fields has been disproven by Li \cite[Theorem 1.3]{LiCentralVanishing}.  More precisely, she finds infinitely many quadratic characters $\chi_D$ such that $L(\tfrac12,\chi_D)=0$ and provides a lower bound for the number of these counterexamples. Based on the size of these lower bounds, it is still possible that Chowla's conjecture holds for {\it almost all} quadratic characters, that is for $100\%$ of quadratic characters. Indeed, Bui and Florea \cite{BF1} prove that $94.29\%$ of $D\in \mathcal{H}_{2g+1}$ satisfy $L(\frac12,\chi_D)\neq 0$ as $g\to \infty$.  

Now, we require an analogue of Hypothesis \ref{Houghhyp}. The justification for this analogue is given in Section \ref{specsec}.
 
 \begin{hyp}[Low-lying Zero Hypothesis]\label{low lying zero hypothesis}
Let $\theta_{j,D}$ be the eigenphases associated to \eqref{Lstardef}. If $y=y(\delta)\to \infty$ then as $\delta\to \infty$ we obtain
\[
\frac{1}{|\mathcal{H}_{n}|}\left|\left\{D\in \mathcal{H}_{n}: \min_{j}|\theta_{j, D}|< \frac{1}{y\delta} \right\}\right|=o(1),
\]
where $\delta$ is defined by \eqref{delta}.
\end{hyp}

\noindent

From here we obtain:
\begin{theorem}\label{main CLT theorem}
Suppose that the Low-lying Zero Hypothesis holds for $\{L(s, \chi_D)\}_{D\in \mathcal{H}_n}$. For $D\in \mathcal{H}_n$, we consider
\[
\widetilde{A}(D)=\frac{1}{\sqrt{\log n}}\left(\log\left| L\left(\frac{1}{2}, \chi_D\right)\right|-\frac{1}{2}\log n\right).
\]
Then we have
\[
\frac{1}{|\mathcal{H}_n|}\sum_{D\in \mathcal{H}_n}\Delta_{\widetilde{A}(D)}\to N(0, 1), \quad n\to \infty.
\]
\end{theorem}

The remark following Theorem \ref{Unconditional CLT near to half line} also applies to Theorem \ref{main CLT theorem}. That is, we have a uniform statement about the distribution of $\log| L(\tfrac12,\chi_D)|$ regardless of whether $n=2g+1$ or $n=2g+2$.\\

\noindent
The plan for the paper is as follows: In section \ref{prelimsec}, we will give necessary background for studying the question over function fields. In section \ref{Lemmata}, we compile a list of useful lemmas to complete the proofs. In Section \ref{logformula}, we will develop an expression for $\log L(\tfrac12+\sigma_0,\chi_D)$ which makes it easier to do the necessary moment calculations. In Section \ref{thm1.1proof}, we prove Theorem \ref{Unconditional CLT near to half line}. In Section \ref{main CLT theorem}, we discuss the discrepancy between $\log|L(\tfrac12,\chi_D)|$ and $\log|L(\tfrac12+\sigma_0),\chi_D)|$ and complete the proof of Theorem \ref{main CLT theorem}. In the final section we prove Corollary \ref{Unconditional CLT}. 

\section{Background for $L$-functions over function fields}\label{prelimsec}
We begin by fixing some notation which will be used throughout the paper.  We will use \cite{ROS} as a general reference.

\subsection{Notations and basics:} Let $q=p^e$, for p a fixed odd prime and $e\ge1$ an integer.  Then let $\mathbb{F}_q$ be the finite field with $q$ elements. The polynomial ring $\mathbb{F}_q[t]$ has many things in common with the integers including satisfying a ``prime number theorem'' for its monic irreducible polynomials. For $f$ in $\mathbb{F}_q[t]$ we denote the degree of the polynomial as $d(f)$ or $\deg(f)$.  The norm of a polynomial $f\in \mathbb{F}_{q}[t]$ is, for $f\neq 0$, define to be $|f|=q^{d(f)}$ and for $f=0$, $|f|=0$.

  We define 
 \begin{align*}
&  \mathcal{M}_{n}=\{f\in\mathbb{F}_q[t] : f \text{ is monic and } d(f) =n\},  \hspace{2mm } \mathcal{M}_{_\le n}=\bigcup_{j\le n}   \mathcal{M}_{j}, \\ 
 & \mathcal{P}_{n}=\{f\in\mathbb{F}_q[t] : f \text{ is monic, irreducible and } d(f) =n\},   \mathcal{P}_{\le n} = \bigcup_{j\le n} \mathcal{P}_{j}, 
 \end{align*}
 and
 \begin{equation*}
 \mathcal{H}_{n}=\{f\in\mathbb{F}_q[t] : f \text{ is monic, square-free and } d(f) =n\}.
 \end{equation*}

\noindent
Observe that $|\mathcal{M}_{n}|=q^{n}$ and for $n\geq 1$, $|\mathcal{H}_{n}|=q^{n-1} (q-1)$. \\

\noindent Let $\Lambda(f)$ be the analogue of the Von Mangoldt function: 
\[\Lambda(f)=\begin{cases}
\deg P & \text{if } f=P^k, P\in\mathcal{P},\\
0 & \text{otherwise }.
\end{cases}
\]
The prime polynomial theorem (see \cite{ROS}, Theorem $2.2$) states that 
\begin{align}\label{prime poly th}
|\mathcal{P}_{n}|=\frac{q^{n}}{n} \,\,+\,\,O\Big(\frac{q^{\frac{n}{2}}}{n}\Big).
\end{align}

\vspace{2mm}
\noindent
 The zeta function of $A=\mathbb{F}_{q}[t]$, denoted by $\zeta_{A}(s)$, is defined by $$\zeta_{A}(s)=\sum_{f\in \mathcal{M}}\frac{1}{|f|^{s}}=\prod_{P\in \mathcal{P}} \left( 1- |P|^{-s} \right)^{-1},$$ for $\Re(s)>1$. One can show that $\zeta_{A}(s)=\frac{1}{1-q^{1-s}}$, and this provides an analytic continuation of zeta function to the complex plane with a simple pole at $s=1$. Using the change of variable $u=q^{-s}$, 
 \begin{align*}
 \zeta_{A}(s)=\sum_{f\in \mathcal{M}}u^{d(f)}=\frac{1}{1-qu}, \quad \text{ if } |u|<\frac{1}{q}.
 \end{align*} 
 \noindent
\subsection{Quadratic Dirichlet characters and their $L$-functions:}
Let $P$ be a monic irreducible polynomial, we define the quadratic character $\left(\frac{f}{P} \right)$ by 
 \begin{align*}
  \left( \frac{f}{P}\right)= 
 \left\{
 \begin{array}
 [c]{ll}
  1 & \text{if\, $f$ is a square\;} (\text{mod}\,\,  P)\,\, , P\nmid f \\
 -1 & \text{if\, $f$ is not a square\;} (\text{mod}\,\, P)\,\, , P\nmid f\\
  0 & \text{if \; } P\mid f.
  \end{array}
 \right.
 \end{align*}
We extend this definition to any $D\in \mathbb{F}_{q}[t]$ multiplicatively and define the quadratic character $\chi_{D}(f)$ as $\left(\frac{D}{f} \right).$ Given any character one can define an $L$-function associated to it:  
 \begin{align*}
 L(s,\chi_{D})=\sum_{f\in \mathcal{M}} \frac{\chi_{D}(f)}{|f|^{s}}=\prod_{P\in \mathcal{P}}\left(1-\chi_{D}(P)\,|P|^{-s} \right)^{-1},\;\; \Re(s)>1.
 \end{align*} 
Using the change of variable $u=q^{-s}$, we have  
 	\begin{align}\label{L-function series form}
 	\mathcal{L}(u,\chi_{D})=\sum_{f\in \mathcal{M} } \chi_{D}(f)\, u^{d(f)}=\prod_{P\in\mathcal{P}}\left(1-\chi_{D}(P)\,u^{d(P)} \right)^{-1},\;\;\, |u|<\frac{1}{q} .
 	\end{align}
 	\noindent
Observe that if $n\geq d(D)$, then $$\sum_{f\in \mathcal{M}_{n}}\chi_{D}(f)=0 .$$ Thus we have $\mathcal{L}(u,\chi_{D})$ is a polynomial of degree at most $d(D)-1$. From here on, we consider $D$ as a monic, square-free polynomial. Then $\mathcal{L}(u, \chi_D)$ has a trivial zero at $u=1$ if and only if $d(D)$ is even. Thus
\begin{align}\label{L-function polynomial form}
L(s,\chi_{D})=\mathcal{L}(u, \chi_D)=(1-u)^{\lambda}\mathcal{L}^{*}(u, \chi_D)=(1-q^{-s})^{\lambda}{L}^{*}(s, \chi_D),
\end{align}
where 
\begin{align}\label{definition of lambda}
  \lambda= 
 \left\{
 \begin{array}
 [c]{ll}
  1 & \text{if\, $d(D)$ even} \\
 0 & \text{if\, $d(D)$ odd},
  \end{array}
 \right.
\end{align}
and $\mathcal{L}^*(u, \chi_D)$ is a polynomial of degree 
\begin{align*} 
2\delta=d(D)-1-\lambda,
\end{align*}
satisfying the functional equation 
\[
\mathcal{L}^*(u, \chi_D)=(qu^2)^{\delta}\mathcal{L}^*\left(\frac{1}{qu}, \chi_D\right).
\]
Because $\mathcal{L}$ and $\mathcal{L}^*$ are polynomial in $u$, it is convenient to define 
\[
L^*(s, \chi_D)=\mathcal{L}^*(u, \chi_D)
\]
so that the above functional equation can be rewritten as
\[
L^*(s, \chi_D)=q^{(1-2s)\delta}L^*(1-s, \chi_D).
\]
\noindent
Additionally, since $\mathcal{L}^*$ is a polynomial, it can be written as a product of its zeros: 
 \begin{equation} \label{Lstardef}
\mathcal{L}^{*}(u, \chi_{D})=\prod_{j=1}^{2\delta}\left(1-u \sqrt{q}\, \alpha_{j}\, \right),
\end{equation}
 where  $\alpha_j=e(-\theta_{j,D})$ are the reciprocals of $u_j=q^{-\frac12}e(\theta_{j,D})$, the roots of $\mathcal{L}^{*}$. We call $\theta_{j,D}$ the eigenphases of the polynomial and are described in more detail in Section \ref{specsec}.
The Riemann Hypothesis, proved by Weil \cite{WEIL} is that the zeros of $\mathcal{L}^*(u, \chi_D)$ all lie on the circle $|u|=q^{-\frac{1}{2}}$.

\vspace{2mm}
\noindent
We define the completed $L$-function in the following way. Set $X_D (s)=|D|^{\frac{1}{2}-s}X(s)$, where
\begin{align*}
X(s)=\left\{
 \begin{array}
 [c]{ll}
  q^{s-\frac{1}{2}} & \text{if\, $d(D)$ odd} \\
 \frac{1-q^{-s}}{1-q^{-(1-s)}}q^{-1+2s} & \text{if\, $d(D)$ even}.
  \end{array}
 \right.
\end{align*}
Let us consider
\begin{align*}
\Lambda(s, \chi_D)=L(s, \chi_D)X_D(s)^{-\frac{1}{2}}.
\end{align*}

\vspace{2mm}
\noindent
Then the above completed $L$-function satisfies the symmetric functional equation 
\begin{align*}
\Lambda(s, \chi_D)=\Lambda(1-s, \chi_D).
\end{align*}

\vspace{2mm}
\noindent
\subsection{The logarithmic derivative:} 

\noindent
By taking logarithmic derivative of  \eqref{L-function series form} and \eqref{L-function polynomial form}  we obtain
	\[
	\frac{\mathcal{L}'}{\mathcal{L}}(u, \chi_D)=\sum_{n=1}^{\infty}\left(\sum_{f\in \mathcal{M}_{n}}\Lambda(f)\chi_D(f)\right)u^{n-1}
	\]
	and 
	\[
	\frac{\mathcal{L}'}{\mathcal{L}}(u, \chi_D)=\lambda (1-u)^{-1}-\sum_{n=1}^{\infty}\left(\sum_{n=1}^{\infty}\alpha_j^n\right)u^{n-1},
	\]
	where $\lambda$ is defined by \eqref{definition of lambda}.
	Equating these two expressions, we obtain
	\[
	\sum_{f\in \mathcal{M}_{n}}\Lambda(f)\chi_D(f)=-\lambda-q^{\frac{n}{2}}\sum_{n=1}^{2\delta}e(-n\theta_j).
	\]
	
	\vspace{2mm}
	\noindent
We may also express the logarithmic derivative in two ways, one in terms of the zeros of $L(s,\chi_D)$ as
\begin{align}\label{log derivative zeros}
 \frac{L'}{L}(s, \chi_D)&=\log q \bigg(\frac{\lambda q^{-s}}{1-q^{-s}}+\sum_{j=1}^{2\delta}\frac{\alpha_j q^{\frac12-s}}{1-\alpha_j q^{\frac12-s}}\bigg)
\end{align}
and the other in terms of the Dirichlet series as
 \begin{align}\label{log derivative Dirichlet}
 \frac{L'}{L}(s, \chi_D)=-\log q \sum_{f\in \mathcal{M}_n}\frac{\Lambda(f)\chi_D(f)}{|f|^s} .
 \end{align}

\subsection{Spectral Interpretation}\label{specsec}
Let $C$ be a non-singular projective curve over $\mathbb{F}_q$ of genus $\delta$. For each extension field of
degree $n$ of $\mathbb{F}_q$, denote by $N_n(C)$ the number of points of $C$ in $\mathbb{F}_{q^n}$. Then, the zeta
function associated to $C$ defined as
\[
Z_C(u)=\exp\left(\sum_{n=1}^\infty N_n(C)\frac{u^n}{n}\right), \quad |u|<\frac{1}{q},
\]
is known to be a rational function of $u$ of the form
\[
Z_C(u)=\frac{P_C(u)}{(1-u)(1-qu)}.
\]
Additionally, we know, $P_C(u)$ is a polynomial of degree $2\delta$ with integer coefficients, satisfying
a functional equation
\[
P_C(u)=(qu^2)^\delta
P_C\left(\frac{1}{qu}\right),
\]
where $\delta$ is defined as in \eqref{delta}.

\vspace{2mm}
\noindent
The Riemann Hypothesis, proved by Weil \cite{WEIL}, says that the zeros of $P_C(u)$ all lie on the circle $|u| =\frac{1}{\sqrt{q}}$. Thus one may give a spectral interpretation of
$P_C(u)$ as the characteristic polynomial of a $2\delta \times 2\delta$ unitary matrix $\Theta_C$:
\[
P_C(u)=\det \left(I- u\sqrt{q}\Theta_C\right).
\]
Thus the eigenvalues $e^{i\theta_j}$ of $\Theta_C$ correspond to the zeros, $q^{-1/2}e^{-i\theta_j}$, of $Z_C(u)$.
The matrix $\Theta_C$ is called the unitarized  Frobenius class of $C$.\\

\noindent
To put this in the context of our case, note that, for a family of hyperelliptic curves $C_D:\, y^2=D(x)$ of genus $\delta$, the numerator of the zeta function $Z_C(u)$  associated to $C_D$ is coincide with the $L$-function $\mathcal{L}^*(u, \chi_D)$, i.e., $P_C(u)=\mathcal{L}^*(u, \chi_D)$.\\

\noindent
It is an interesting problem to study how the Frobenius classes $\Theta_C$ change as we vary the associated curve over a family of hyperelliptic curves with genus $\delta$. As mentioned in the introduction, there are three aspects where we can study the distribution of of these Frobenius classes. 

\subsubsection*{(i) Large finite field aspect}
Katz and Sarnak \cite{KatzandSarnak} showed that that as $q\to \infty$, the Frobenius classes $\Theta_D$
become equidistributed in the unitary symplectic group $USp(2\delta)$. More precisely,
\begin{equation}\label{KSqlim}
\lim_{q\to \infty}\frac{1}{|\mathcal{H}_n|}\sum_{D\in \mathcal{H}_n}F(\Theta_D)=\int_{USp(2\delta)}F(U)dU,
\end{equation}
for any continuous function $F$ on the space of conjugacy classes of $USp(2\delta)$.
This implies that in the large finite field aspect various statistics of the eigenvalues can be
computed by integrating the corresponding quantities over $USp(2\delta).$
\subsubsection*{(iii) Double limit aspect} See \cite[pg. 11]{KatzandSarnak} where they let $\delta\to\infty$ in \eqref{KSqlim}.
\subsubsection*{(ii) Large degree aspect}
 Since the matrices $\Theta_D$ now inhabit
different spaces as $\delta$ grows, it is not clear how to formulate an equidistribution problem. 
 The following analysis of Katz and Sarnak \cite{KatzSarnak2} illuminates one possible interpretation. We start with an even test function $f$, say, in the Schwartz space $\mathcal{S}(\mathbb{R})$, and for any $N \geq  1$ set
\[
F(\theta):= \sum_{k\in \mathbb{Z}}f\left(N\left(\frac{\theta}{2\pi }-k\right)\right).
\] 
$F(\theta)$ has period $2\pi$ and is localized in an interval of size $\thickapprox 1/N \in \mathbb{R}/2\pi \mathbb{Z}$.
Next, for a unitary $2\delta \times 2\delta$ matrix $U$ with eigenvalues $e^{i\theta_j}, j = 1, \ldots,2\delta$, define
\[
Z_f(U):= \sum_{j=1}^{2\delta}F(\theta_{j, D}).
\]

\vspace{2mm}
\noindent
Now,  $Z_f(U)$ counts the number of ``low-lying'' eigenphases, $\theta_{j,D}$, in the smooth interval of length $\thickapprox \frac1N$ around the origin defined by $f$. In other words, for $j \geq 1$, the above discussion describes the distribution of the numbers
\[
\frac{\theta_{j, D}\, \delta}{2\pi}
\]
as $D$ varies over $\mathcal{H}_n$, $\, n\to \infty$. That is, we use this to study the distribution of the $j$-th lowest zero.
\begin{conjecture}[Density Conjecture]\label{conjecture of katz and sarnak}
For a fixed $q$, we have
\[
\lim_{n\to \infty}\frac{1}{|\mathcal{H}_n|}\sum_{D\in \mathcal{H}_n}Z_f(\Theta_D)=\int_{USp(2\delta)}Z_f(U)dU
\]
for any test function defined as above.
\end{conjecture}

\noindent
It is also known that 
\[
\lim_{\delta\to \infty}\int_{USp(2\delta)}Z_f(U)dU\sim \int_{-\infty}^{\infty}f(x)\left(1-\frac{\sin 2\pi x}{2\pi x}\right)dx.
\]

\vspace{2mm}
\noindent
Rudnick \cite{Rudnick} proved the Conjecture \ref{conjecture of katz and sarnak} for  $f \in  \mathcal{S}(\mathbb{R})$ is even, with Fourier transform $\hat{f}$ supported in $(-2, 2)$ before Bui and Florea \cite{BF1}  proved the same with the support $(1/4, 1/2)$.\\

\noindent
As an application, the Density Conjecture \ref{conjecture of katz and sarnak} gives the distribution of zeros for a family $L(s, \chi_D)$ near $s =\frac{1}{2}$. As discussed above, this has immediate applications to counting how often $L(\tfrac12,\chi_D)=0$. More precisely, by varying the test function $f$ in the Density Conjecture for any of the above family of hyperelliptic curves $C_D$, one is led to (assuming
the Density Conjecture) (See equation $(56)$ of \cite{KatzSarnak2}):
\begin{align*} 
\lim_{n\to \infty}\frac{1}{|\mathcal{H}_n|}\# \{D\in \mathcal{H}_n : L(1/2, \chi_D)\neq 0\}=1.
\end{align*}

\noindent
Thus, it is clear that, for the above family of hyperelliptic curves, Hypothesis \ref{low lying zero hypothesis} is a consequence of Conjecture \ref{conjecture of katz and sarnak}.

\section{Preliminary lemmas}\label{Lemmata}
This section is simply a collection of lemmas necessary for the final result. 
	\begin{lemma}\label{upper bound for moment}
	Let $k, y$ be integers such that $2ky\leq n$. For any sequence of complex number $\{a(P)\}_{P\in\mathcal{P}}$, we have
	\[
	\sum_{D\in \mathcal{H}_{n}}\bigg|\sum_{d(P)\leq y}\frac{\chi_D(P)a(P)}{|P|^{\frac{1}{2}}}\bigg|^{2k}\ll q^{n}\frac{(2k)!}{k! 2^k}\left(\sum_{d(P)\leq y}\frac{|a(P)|^2}{|P|}\right)^k.
	\]
	\end{lemma}
	\begin{proof}
The case $n=2g+1$  is proved by Florea, \cite[Lemma 8.4]{Floriafourthmoment}. To get the result for $n=2g+2$ it is a small adaptation of Florea's proof. 
	\end{proof}

	\begin{lemma}\label{square term evaluation}
	Let $f\in \mathcal{M}_n$. Then 
	\[
	\frac{1}{|\mathcal{H}_{n}|}\sum_{D\in \mathcal{H}_{n}}\chi_D(f^2)=\prod_{P|f}\left(1+\frac{1}{|P|}\right)^{-1}+O(q^{-n-1}).
	\]
	\end{lemma}
	\begin{proof}
See [\cite{BF1}, Lemma $3.7$] for $2g+1$ case.  To get the result for $n=2g+2$ it is a small adaptation of their proof.
		\end{proof}
	\begin{lemma}[P\'{o}lya-Vinogradov inequality]\label{polya-vinogradov inequality}
	For $l\in \mathcal{M}_n$ not a square polynomial, let $l=l_1l_2^2$ with $l_1$ square free. Then we get
	\[
	\bigg|\sum_{D\in \mathcal{H}_{n}}\chi_D(l)\bigg|\ll_{\epsilon} q^g|l_1|^{\epsilon}.
	\]
	\end{lemma}
	\begin{proof}
	See [\cite{BF2}, Lemma $3.5$] for the case $n=2g+1$. To obtain the result for $n=2g+2$ follow the same proof, there is only a minor difference to handle the extra term coming from the trivial zero $u^{-1}$. 
	\end{proof}
	
	\begin{lemma}\label{upper bound for truncated sum over primes }
	Let $K\geq 2$ and $\sigma_o<1$ be such that $K\sigma_o<\frac{1}{2\log q}$. Then we have
	\[
	\sum_{d(P)\leq K}\frac{1}{|P|^{1+2\sigma_o}}=\log K +O(1)+O\left(q^{-\frac{K}{4}}\right).
	\]
	\end{lemma}
	\begin{proof}
From the prime polynomial theorem \eqref{prime poly th} and partial summation formula, we obtain
\begin{align*}
\sum_{d(P)\leq K}\frac{1}{|P|^{1+2\sigma_o}}&=\sum_{n\leq K}\frac{1}{nq^{2n\sigma_o}}+ O\bigg(\sum_{n\leq K}\frac{1}{nq^{n/2}}\bigg)\\
&=\log K+ O(1)+O\left(q^{-\frac{K}{4}}\right).
\end{align*}
	\end{proof}
	\noindent
	Now we state a standard probabilistic lemma which tells us when the measure of a certain set is very tiny. 
	\begin{lemma}\label{small measurable sets}
	Let $R$ be random variable defined on the subspace $\mathcal{H}_n$ of $\mathcal{M}_n$.   Suppose that the second moment of $R$ is small in the sense of 
	\begin{align*}
\sum_{D\in \mathcal{H}_n}|R_D|^2=O\left(|\mathcal{H}_n|\right).
\end{align*}
Then
\[\#\{D\in \mathcal{H}_n : R_D>T\} \le \frac{|\mathcal{H}_n|}{T}.\]
	\end{lemma}
	
\subsection{Estimates of Sums involving $\Lambda_X$}
Define 
\begin{equation}\label{Lambda X def} 
\Lambda_X(f)=\begin{cases}
2X^2\Lambda(f) & \text{ if } d(f) \le X \\ 
(X^2-(d(f))^2+2d(f)X-3X-d(f)-2)\Lambda(f) & \text{ if } X< d(f) \le 2X \\ 
(3X-d(f)+1)(3X-d(f)+2)\Lambda(f) & \text{ if } 2X <d(f) \le 3X.
\end{cases}
\end{equation}
We have the following estimates: 

\begin{lemma}\label{mean square estimate} Let $k\ge 1$ be an integer and $X$ be a real number such that $X< \frac{n}{3k}$. Suppose that $\sigma_0>0$. Then 
\[\frac1{|\mathcal{H}_n|}\sum_{D\in \mathcal{H}_n}\left|\frac1{X^3}\sum_{f\in \mathcal{M}_{\leq 3X}}\frac{\Lambda_X(f)\chi_D(f)}{|f|^{\tfrac12+\sigma_0}}\right|^k=O\left(\left(\frac{18k}{e}\right)^{\frac{k}2}+\frac{4^k}{X^k}\right).\]
\end{lemma}
\begin{remark}
In particular, taking $k=2$ we will obtain that this expression is $O(1)$. 
\end{remark}

\begin{proof}
From the definition of $\Lambda _X(f)$, we can expand this into a sum over prime powers so that 

\begin{align*}
&\frac1{|\mathcal{H}_n|}\sum_{D\in \mathcal{H}_n}\left|\frac1{X^3}\sum_{f\in \mathcal{M}_{\leq 3X}}\frac{\Lambda_X(f)\chi_D(f)}{|f|^{\tfrac12+\sigma_0}}\right|^k \\ 
& \ll \frac1{|\mathcal{H}_n|}\sum_{D\in \mathcal{H}_n}\left|\frac1{X^3}\left(\sum_{P\in \mathcal{M}_{\leq 3X}}\frac{\Lambda_X(P)\chi_D(P)}{|P|^{\tfrac12+\sigma_0}}+\sum_{P^2\in \mathcal{M}_{\leq 3X}}\frac{\Lambda_X(P^2)\chi_D(P^2)}{|P^2|^{\tfrac12+\sigma_0}}\right)\right|^k \\
&\ll2^k\left( \frac1{|\mathcal{H}_n|}\sum_{D\in \mathcal{H}_n}\left|\frac1{X^3}\sum_{P\in \mathcal{M}_{\leq 3X}}\frac{\Lambda_X(P)\chi_D(P)}{|P|^{\tfrac12+\sigma_0}}\right|^k+ \frac1{|\mathcal{H}_n|}\sum_{D\in \mathcal{H}_n}\left|\frac1{X^3}\sum_{P^2\in \mathcal{M}_{\leq 3X}}\frac{\Lambda_X(P^2)\chi_D(P^2)}{|P^2|^{\tfrac12+\sigma_0}}\right|^k\right) \\
&\ll 2^k(S_1 +S_2).
\end{align*}
We first consider $S_2$:  By definition,  $\chi_D(P^2)=1$, and the $\Lambda_X(P^2) \le 4X^2d(P)$ so we have 
\[
S_2\ll \frac1{|\mathcal{H}_n|}\sum_{D\in \mathcal{H}_n}\left|\frac4{X}\sum_{P^2\in\mathcal{M}_{\leq 3X}}\frac{d(P)}{|P|^{1+2\sigma_0}}\right|^k.
\] 

\vspace{2mm}
\noindent
From the prime polynomial theorem we obtain 
\[ S_2\ll \frac1{|\mathcal{H}_n|}\sum_{D\in \mathcal{H}_n}\left|\frac4{X}\sum_{2m\le 3X}\frac{m q^m}{m q^{m(1+2\sigma_0)}}\right|^k\ll 6^k.\]
For $S_1$ we expand the power k and then swap the order of summation giving 

\[
S_1 =\frac1{|\mathcal{H}_n|}\frac1{X^{3k}}\sum_{P_1,P_2,\ldots, P_k \in\mathcal{M}_{\leq 3X}}\frac{\Lambda_X(P_1)\Lambda_X(P_2)\cdots\Lambda_X(P_k)}{|P_1P_2\cdots P_k|^{\tfrac12+\sigma_0}} \sum_{D\in \mathcal{H}_n}\chi_D(P_1P_2\cdots P_k).
\]
From Lemmas \ref{square term evaluation} and \ref{polya-vinogradov inequality} we see that the values are different depending on if $P_1P_2\cdots P_k$ is a square or not. Thus we consider two cases: 

\vspace{2mm}
\noindent
Case 1: $k$ is even  and $P_1P_2\cdots P_k=\square$.\\
Applying Lemma \ref{square term evaluation}, we must estimate 
\[\frac1{X^{3k}}\sum_{\substack{P_1,P_2,\ldots,P_k\in\mathcal{M}_{\leq 3X}\\ P_1P_2\cdots P_k=\square}}\frac{\Lambda_X(P_1)\Lambda_X(P_2)\cdots\Lambda_X(P_k)}{|P_1P_2\cdots P_k|^{1/2+\sigma_0}}\left(1+\frac1{|P_1P_2\cdots P_k|}\right)^{-1}.\]
First note that $P_1P_2\cdots P_k=\square$ happens precisely when one can pair up the indices so that the corresponding primes are equal thus using that  $\Lambda_X(P)=O(4X^2d(P))$  we have the above
\begin{align*}
\ll& \frac1{X^{3k}}\frac{k!}{(k/2)!2^{k/2}}\left(\sum_{P\in\mathcal{M}_{\leq 3X}}\frac{(\Lambda_X(P))^2}{|P|^{1+2\sigma_0}}\left(1+\frac1{|P|}\right)^{-1}\right)^{k/2}\\
&\ll \frac{k!2^{k/2}}{(k/2)!}\frac1{X^k}\left(\sum_{P\in\mathcal{M}_{\leq 3X}}\frac{(d(P))^2}{|P|^{1+2\sigma_0}}\left(\frac{|P|}{|P|+1}\right)\right)^{k/2}.
\end{align*}
Then, using the prime polynomial theorem and Stirling's approximation, we obtain
\begin{align*}
&\ll \frac{k!2^{k/2}}{(k/2)!}\frac1{X^k}\left(\sum_{P\in\mathcal{M}_{\leq 3X}}\frac{m^2 q^m}{mq^m}\right)^{k/2}\\ 
& \ll \frac{k!2^{k/2}}{(k/2)!}\frac{((3X)(3X+1))^{k/2}}{2^{k/2}X^k} \ll \left(\frac{18k}{e}\right)^{k/2}. 
\end{align*}

\vspace{2mm}
\noindent
Case 2: $P_1P_2\cdots P_k \neq\square$. \\ 
In this case, using Lemma \ref{polya-vinogradov inequality}, we need to estimate: 
\[\frac1{\sqrt{|\mathcal{H}_n|}}\frac1{X^{3k}}\sum_{\substack{P_1,P_2,\ldots,P_k\in\mathcal{M}_{\leq 3X}\\ P_1P_2\cdots P_k \neq\square }}\frac{\Lambda_X(P_1)\Lambda_X(P_2)\cdots \Lambda_X(P_k)}{|P_1P_2\cdots P_k|^{\tfrac12+\sigma_0}}|P_1P_2\cdots P_k|^{\epsilon}.\]
Using that  $\Lambda_X(P)=O(X^2d(P))$ we find
\begin{align*}
&\frac1{\sqrt{|\mathcal{H}_n|}}\frac1{X^{3k}}\sum_{\substack{P_1,P_2,\ldots, P_k\in\mathcal{M}_{\leq 3X}\\ P_1P_2\cdots P_k \neq\square }}\frac{\Lambda_X(P_1)\Lambda_X(P_2)\cdots \Lambda_X(P_k)}{|P_1P_2\cdots P_k|^{\tfrac12+\sigma_0}}|P_1P_2\cdots P_k|^{\epsilon} \\
&\ll 
\frac1{\sqrt{|\mathcal{H}_n|}}\frac{4^k}{X^k}\sum_{\substack{P_1,P_2,\ldots,P_k\in\mathcal{M}_{\leq 3X}\\P_1P_2\cdots P_k \neq\square}}\frac{d(P_1)d(P_2)\cdots d(P_k)}{|P_1P_2\cdots P_k|^{\tfrac12+\sigma_0-\epsilon}} \\
&\ll \frac1{\sqrt{|\mathcal{H}_n|}}\frac{4^k}{X^k}\left(\sum_{P\in\mathcal{M}_{\leq 3X}}\frac{d(P)}{|P|^{\tfrac12+\sigma_0-\epsilon}}\right)^k \\
&\ll \frac1{\sqrt{|\mathcal{H}_n|}}\frac{4^k}{X^k}\left(\sum_{m\le 3X}q^{m(\tfrac12+\epsilon)}\right)^k \\
&\ll \frac{4^kq^{3kX(\tfrac12+\epsilon)}}{q^{n/2}X^k}\ll (4/X)^k.
\end{align*}
To see that this is indeed $\ll (4/X)^k$ we use the assumption  $X < n/3k$. 
Combining these estimates we have the desired result. 
\end{proof}
\subsection{Estimates for sums over the zeros}\label{LprimeLisreal}

We may rewrite \eqref{log derivative zeros} as 
\begin{equation}\label{log deriv version 4}
\frac{L'}{L}(s,\chi_D)=\log q \left(\frac{\lambda q^{-s}}{1-q^{-s}}-2\delta +\sum_{j=1}^{2\delta}\left(\frac1{1-\alpha_jq^{1/2-s}}\right)\right).
\end{equation}
Notice that the non-trivial zeros appear in complex conjugate pairs and so \eqref{log deriv version 4} gives a real value for $s\in \mathbb{R}$, as was pointed out in \cite{AT}. Thus, we may also express \eqref{log deriv version 4} as follows: 
\begin{equation*}
\frac{L'}{L}(s,\chi_D)=\log q \left(\frac{\lambda q^{-s}}{1-q^{-s}}-\delta +\sum_{j=1}^{2\delta}\Re\left(\frac1{1-\alpha_jq^{1/2-s}}-\frac12\right)\right).
\end{equation*}

\begin{lemma}[Lemma 3.2, \cite{AT}]\label{TsimAltu}
Let $\sigma_0>0$. Then for $\sigma\ge\frac12+\sigma_0$, we have the inequality  
\[\left|\frac{\alpha_j^{-1}q^{\sigma-1/2}}{(1-\alpha_j^{-1} q^{\sigma-1/2})^{2}}\right| \le \frac{(\sigma-\tfrac12)^{-1}}{\log q}\Re\left(\frac1{1-\alpha_jq^{\frac12-\sigma}}\right)\le \frac1{\sigma_0\log q}\Re\left(\frac1{1-\alpha_jq^{-\sigma_0}}\right).\]
\end{lemma}
\begin{lemma}\label{bnd-sum-zeros}
Let $\sigma_0>0$. Then for $\sigma\ge \frac12+\sigma_0$,
\[\sum_{j=1}^{2\delta}\left|\frac{(\alpha_jq^{\frac12-\sigma})^X(1-(\alpha_jq^{\frac12-\sigma})^X)^2}{(1-\alpha_j^{-1}q^{\sigma-\frac12})^3}\right|\le \frac{4q^{X(\frac12-\sigma)}}{\sigma_0^2\log^3 q}\left(\frac{L'}{L}\left(\frac12+\sigma_0,\chi_D\right)+2\delta\log q\right).\]
\end{lemma}
\begin{proof}
We begin with 
\[\sum_{j=1}^{2\delta}\frac{(\alpha_jq^{\frac12-\sigma})^X(1-(\alpha_jq^{\frac12-\sigma})^X)^2}{(1-\alpha_j^{-1}q^{\sigma-\frac12})^3}.\]
Lemma \ref{TsimAltu} gives
\begin{multline*}
\left|\frac{(\alpha_jq^{\frac12-\sigma})^{X+1}(1-(\alpha_jq^{\frac12-\sigma})^X)^2(\alpha_jq^{\frac12-\sigma})^{-1}}{(1-\alpha_j^{-1}q^{\sigma-\frac12})(1-\alpha_j^{-1}q^{\sigma-\frac12})^2}\right| \\
\le\left|\frac{(\alpha_jq^{\frac12-\sigma})^{X+1}(1-(\alpha_jq^{\frac12-\sigma})^X)^2}{(1-\alpha_j^{-1}q^{\sigma-\frac12})}\right|\frac1{\sigma_0\log q}\Re\left(\frac1{1-\alpha_jq^{-\sigma_0}}\right)\end{multline*}

\vspace{2mm}
\noindent
We may easily bound the numerator as $|(\alpha_jq^{\frac12-\sigma})^{X+1}(1-(\alpha_jq^{\frac12-\sigma})^X)^2| \le 4q^{X(\frac12-\sigma)}$. It remains to discuss the denominator:

\noindent
We have $1-\alpha_j^{-1}q^{\sigma-\frac12} = 1-\cos(2\pi\theta_j)q^{\sigma-\frac12}-i\sin(2\pi\theta_j)q^{\sigma-\frac12}$ and so 
\begin{align*}
|1-\alpha_j^{-1}q^{\sigma-\frac12}|& =\sqrt{(1-\cos(2\pi\theta_j)q^{\sigma-\frac12})^2+(\sin(2\pi\theta_j)q^{\sigma-\frac12})^2}\\
&=\sqrt{1+q^{2\sigma-1}-2\cos(2\pi\theta_j)q^{\sigma-\frac12}} 
\end{align*}
Since $\sigma\ge \frac12+\sigma_0$, we have 
\[\frac1{|1-\alpha_j^{-1}q^{\sigma-\frac12}|}\le \frac1{\sqrt{1+q^{2\sigma_0}-2q^{\sigma_0}}}\le \frac1{\sigma_0\log q},\]
where the last inequality follows from the Taylor expansion.
Thus we have 
\[\sum_{j=1}^{2\delta}\left|\frac{(\alpha_jq^{\frac12-\sigma})^X(1-(\alpha_jq^{\frac12-\sigma})^X)^2}{(1-\alpha_j^{-1}q^{\sigma-\frac12})^3}\right|\le \frac{4q^{X(\frac12-\sigma)}}{(\sigma_0\log q)^2}\sum_{j=1}^{2\delta}\Re\left(\frac1{1-\alpha_jq^{-\sigma_0}}\right).\]

\vspace{2mm}
\noindent
Finally, \eqref{log deriv version 4} gives us that 
\[\sum_{j=1}^{2\delta}\Re\left(\frac1{1-\alpha_jq^{-\sigma_0}}\right)=\frac1{\log q}\frac{L'}{L}\left(\tfrac12+\sigma_0,\chi_D\right)+2\delta -\frac{\lambda q^{-(\frac12+\sigma_0)}}{1-q^{-(\frac12+\sigma_0)}}.\]
Note that $\frac{\lambda q^{-(\frac12+\sigma_0)}}{1-q^{-(\frac12+\sigma_0)}}$ is a positive quantity and it is too small to effect the analysis so we drop it in the upperbound. 
\end{proof}

\section{A formula for $\log L(\sigma,\chi_D)$}\label{logformula}
\noindent
Throughout the rest of the paper, we fix the following notations:  the parameter $\delta$ as defined as in \eqref{delta}, the parameter $\sigma_0=\frac{c}{X}$, with $0<c<\frac1{2\log q}$.
Furthermore, we define 
 \begin{equation}
\widetilde{P}_X(\sigma, \chi_D):=\sum_{f\in \mathcal{M}_{\leq X}}\frac{\Lambda(f)\chi_D(f)}{d(f)|f|^{\sigma}}.
\end{equation}
 The purpose of this section is to prove the following proposition:
\begin{proposition}\label{logL-as-PX}
Let $X\ge1$. Then 
\begin{multline*}
\log L(1/2+\sigma_0,\chi_D)=\widetilde{P}_X(1/2+\sigma_0,\chi_D)
+O\left(\frac{1}{X^3}\sum_{f\in\mathcal{M}_{\leq 3X}}\frac{\Lambda_X(f)\chi_D(f)}{|f|^{\frac12+\sigma_0}}\right)\\+O\left(\frac{\delta}{X}+\frac{\lambda}{X(X+2)}\right).
\end{multline*}
\end{proposition}
\noindent
This proposition is useful as it allows us to compute the moments of $\log|L(\tfrac12+\sigma_0,\chi_D)|$ by studying the simpler expression $\widetilde{P}_X(\sigma, \chi_D)$. The proposition follows immediately from the following lemma after some straightforward manipulations.
\begin{lemma}\label{logLformula1} 
Let $X\ge 1$. Then we have
\begin{align*}
\log L\left(\frac12+\sigma_0,\chi_D\right)=\frac{1}{2X^2}\sum_{f\in\mathcal{M}_{\leq 3X}}\frac{\Lambda_X(f)\chi_D(f)}{d(f)|f|^{\frac12+\sigma_0}}+O\left(\frac{1}{X^3}\sum_{f\in\mathcal{M}_{\leq 3X}}\frac{\Lambda_X(f)\chi_D(f)}{|f|^{\frac12+\sigma_0}}\right)\\
+O\left(\frac{\delta}{X}+ \frac{\lambda}{X(X+2)}\right).
\end{align*}
\end{lemma}

\vspace{2mm}
\noindent
We will make use of the following identity, let $X\ge 0$, $\Re(s)\ge 0$: 
\begin{align}\label{Integral weight}
\frac1{2\pi i}\int_{2}^{2+\frac{2\pi i}{\log q}}\frac{q^{X(w-s)}}{(1-q^{-(w-s)})^3}dw= -\frac{(X+1)(X+2)}{2\log q}. 
\end{align}

\vspace{2mm}
\begin{proof}[Proof of Lemma \ref{logLformula1}]

We begin by considering the following integral 
\begin{align*}
\mathcal{I}&=\frac1{2\pi i}\int_{2}^{2+\frac{2\pi i}{\log q}}\frac{q^{X(w-s)}(1-q^{X(w-s)})^2}{(1-q^{-(w-s)})^3}\frac{L'}{L}(w,\chi_D)dw  \\
&=\frac1{2\pi i}\int_{2}^{2+\frac{2\pi i}{\log q}}\frac{q^{X(w-s)}(1-2q^{X(w-s)}+q^{2X(w-s)})}{(1-q^{-(w-s)})^3}\frac{L'}{L}(w,\chi_D)dw. 
\end{align*}

\vspace{2mm}
\noindent
We compute this integral in two different ways. The first approach is to use \eqref{log derivative Dirichlet}, \eqref{Integral weight} and integrate the result term by term: 
\begin{align*}
\mathcal{I}= \frac12\sum_{f\in \mathcal{M}_{\leq X}}\frac{\Lambda(f)\chi_D(f)}{|f|^s}(X-d(f)+1)(X-d(f)+2) \\
- \sum_{f\in \mathcal{M}_{\leq 2X}}\frac{\Lambda(f)\chi_D(f)}{|f|^s}(2X-d(f)+1)(2X-d(f)+2) \\
+\frac12\sum_{f\in \mathcal{M}_{\leq 3X}}\frac{\Lambda(f)\chi_D(f)}{|f|^s}(3X-d(f)+1)(3X-d(f)+2). 
\end{align*}

\vspace{2mm}
\noindent
We combine these into a single sum recognizing the weight $\Lambda_X(f)$ defined in \eqref{Lambda X def}, so that 
\begin{equation}\label{I version 1}
\mathcal{I}=\frac12\sum_{f \in \mathcal{M}_{\leq 3X}}\frac{\Lambda_X(f)\chi_D(f)}{|f|^s}.
\end{equation}

\vspace{2mm}
\noindent
The second approach is to use the residue theorem, from this method, we see that there is a pole of order 3 when $w=s$ and simple poles for each zero of $\frac{L'}{L}(w,\chi_D)$. So,
\begin{equation}\label{I version 2}
\mathcal{I}=-\frac{X^2}{\log q}\frac{L'}{L}(s,\chi_D) -\log q\bigg(\frac{\lambda q^{-(X+2)s}}{1-q^{-s}}+\sum_{j=1}^{2\delta}\frac{(q^{\frac12-s}\alpha_j)^X(1-(q^{\frac12-s}\alpha_j)^X)^2}{(1-q^{s-\frac12}\alpha_j^{-1})^3}\bigg).
\end{equation}

\vspace{2mm}
\noindent
Equating \eqref{I version 1} and \eqref{I version 2}, we have 

 \begin{multline*}
 \frac12\sum_{f \in \mathcal{M}_{\leq 3X }}\frac{\Lambda_X(f)\chi_D(f)}{|f|^s}\\=-\frac{X^2}{\log q}\frac{L'}{L}(s,\chi_D) -\log q\bigg(\frac{\lambda q^{-(X+2)s}}{1-q^{-s}}+\sum_{j=1}^{2\delta}\frac{(q^{\frac12-s}\alpha_j)^X(1-(q^{\frac12-s}\alpha_j)^X)^2}{(1-q^{s-\frac12}\alpha_j^{-1})^3}\bigg),
 \end{multline*}
which implies 

\begin{align}\label{log derivative version 3}
 -\frac{L'}L(s,\chi_D)=& \frac{\log q}{2X^2}\sum_{f \in \mathcal{M}_{\leq 3X}}\frac{\Lambda_X(f)\chi_D(f)}{|f|^{s}}\\ 
& \nonumber +\frac{(\log q)^2}{2X^2}\bigg(\frac{\lambda q^{-(X+2)s}}{1-q^{-s}}+\sum_{j=1}^{2\delta}\frac{(q^{\frac12-s}\alpha_j)^X(1-(q^{\frac12-s}\alpha_j)^X)^2}{(1-q^{s-\frac12}\alpha_j^{-1})^3}\bigg).
\end{align}
We note that this implies, for $\Re(s)=\sigma\ge \frac12$, we have 
\begin{align}\label{L'/L-intermediate}
-\frac{L'}L(\sigma,\chi_D)&= \frac{\log q}{2X^2}\sum_{f \in \mathcal{M}_{\leq 3X}}\frac{\Lambda_X(f)\chi_D(f)}{|f|^{\sigma}}\\
&\nonumber +\frac{(\log q)^2}{2X^2}\bigg(\frac{\lambda q^{-(X+2)\sigma}}{1-q^{-\sigma}}+\sum_{j=1}^{2\delta}\frac{(\alpha_jq^{\frac12-\sigma})^X(1-(\alpha_jq^{\frac12-\sigma})^X)^2}{(1-\alpha_j^{-1}q^{\sigma-\frac12})^3}\bigg).
\end{align}

\noindent
Now, it remains to provide estimates for the error terms.  We handle the sum over the zeros in \eqref{L'/L-intermediate} by applying Lemma \ref{bnd-sum-zeros}, which gives 
\[\frac{(\log q)^2}{X^2}\sum_{j=1}^{2\delta}\left|\frac{(\alpha_jq^{\frac12-\sigma})^X(1-(\alpha_jq^{\frac12-\sigma})^X)^2}{(1-\alpha_j^{-1}q^{\sigma-\frac12})^3}\right| \le \frac{4 q^{X(\frac12-\sigma)}}{X^2\sigma_0^2\log q}\left(\frac{L'}{L}\left(\frac12+\sigma_0,\chi_D\right)+2\delta \log q\right).\]

Using the fact that  $\sigma_0=\frac{c}{X}$, for $c$ some positive constant, we have 
\[\frac{(\log q)^2}{X^2}\sum_{j=1}^{2\delta}\left|\frac{(\alpha_jq^{\frac12-\sigma})^X(1-(\alpha_jq^{\frac12-\sigma})^X)^2}{(1-\alpha_j^{-1}q^{\sigma-\frac12})^3}\right| \le \frac{4q^{X(\frac12-\sigma)}}{c^2\log q}\left(\frac{L'}{L}\left(\frac12+\sigma_0,\chi_D\right)+2\delta\log q\right).\]

\vspace{2mm}
\noindent
Thus, we can find some $v$ such that $|v|\le 1$ such that
\begin{align}\label{intermediate equation1}
-\frac{L'}L(\sigma,\chi_D)&= \frac{\log q}{2X^2}\sum_{f \in \mathcal{M}_{\leq 3X}}\frac{\Lambda_X(f)\chi_D(f)}{|f|^{\sigma}} \\
\nonumber &+\frac{4\nu q^{X(\frac12-\sigma)}}{c^2\log q}\frac{L'}{L}\left(\frac12+\sigma_0,\chi_D\right)
+O\left(\delta q^{X(\frac12-\sigma)}+ \frac{\lambda q^{-(X+2)\sigma}}{X}\right).
\end{align}

\vspace{2mm}
\noindent
So, evaluating \eqref{intermediate equation1} at $\sigma=\frac12+\sigma_0$, we conclude 
\begin{align*}\left(1+\frac{4\nu q^{-X\sigma_0}}{c^2\log q}\right)\frac{L'}{L}\left(\frac12+\sigma_0,\chi_D\right)= \frac{\log q}{2X^2}\sum_{f \in \mathcal{M}_{\leq 3X}}\frac{\Lambda_X(f)\chi_D(f)}{|f|^{\frac12+\sigma_0}}\\
 +O\left(\delta q^{-X\sigma_0}+ \frac{\lambda q^{-(X+2)(1/2+ \sigma_0)}}{X}\right).
\end{align*}
Additionally, we choose $c$ such that $c^2q^c>\frac{4}{\log q}$, which implies $\left(1+\frac{4\nu q^{-X\sigma_0}}{c^2\log q}\right)>1-\frac4{c^2q^c\log q}>C$, for some $C>0$. Note that this is a valid choice for $c$ since we assumed $c< \frac1{2\log q}$. Thus, we obtain
\begin{equation}\label{LprimeL-identity}
\frac{L'}L\left(\frac12+\sigma_0,\chi_D\right)=O\left(\frac1{X^2}\sum_{f \in \mathcal{M}_{\leq 3X}}\frac{\Lambda_X(f)\chi_D(f)}{|f|^{\frac12+\sigma_0}}\right)+O(\delta).
\end{equation}
Finally, putting together \eqref{LprimeL-identity} and \eqref{intermediate equation1}, we find that
\begin{align*}
-\frac{L'}L(\sigma,\chi_D)&= \frac{\log q}{2X^2}\sum_{f \in \mathcal{M}_{\leq 3X}}\frac{\Lambda_X(f)\chi_D(f)}{|f|^{\sigma}}\\
&+O\left(\frac{q^{X(\frac12-\sigma)}}{X^2}\sum_{f \in \mathcal{M}_{\leq 3X}}\frac{\Lambda_X(f)\chi_D(f)}{|f|^{\frac12+\sigma_0}} \right)
+O\left(\delta q^{X(\frac12-\sigma)}+ \frac{\lambda q^{-(X+2)\sigma}}{X}\right).
\end{align*}
Integrating with respect to $\sigma$ from $\frac12+\sigma_0$ to $\infty$, we conclude that 
\begin{align*}
\log L\left(\frac12+\sigma_0,\chi_D\right)&=\frac{1}{2X^2}\sum_{f \in \mathcal{M}_{\leq 3X}}\frac{\Lambda_X(f)\chi_D(f)}{d(f)|f|^{\frac12+\sigma_0}} \\
&+O\left(\frac1{X^3}\sum_{f \in \mathcal{M}_{\leq 3X}}\frac{\Lambda_X(f)\chi_D(f)}{|f|^{\frac12+\sigma_0}}\right)
+O\left(\frac{\delta}{X}+\frac{\lambda}{X(X+2)}\right).
\end{align*}
\end{proof}
\vspace{2mm}
\noindent
Now, we can complete the proof of Proposition \ref{logL-as-PX}:
\begin{proof}
Simply apply Lemma \ref{logLformula1}: 
\begin{align*}
\log L\left(\frac12+\sigma_0,\chi_D\right)&=\frac{1}{2X^2}\sum_{f\in\mathcal{M}_{\leq 3X}}\frac{\Lambda_X(f)\chi_D(f)}{d(f)|f|^{\frac12+\sigma_0}} \\ 
&+O\left(\frac{1}{X^3}\sum_{f\in\mathcal{M}_{\leq 3X}}\frac{\Lambda_X(f)\chi_D(f)}{|f|^{\frac12+\sigma_0}}\right)+O\left(\frac{\delta}{X}+\frac{\lambda}{X(X+2)}\right)\\
&=\sum_{f\in\mathcal{M}_{\le X}}\frac{\Lambda(f)\chi_D(f)}{d(f)|f|^{\frac12+\sigma_0}}+\frac{1}{2X^2}\sum_{X<d(f)\le 3X}\frac{\Lambda_X(f)\chi_D(f)}{d(f)|f|^{\frac12+\sigma_0}} \\
&+O\left(\frac{1}{X^3}\sum_{f\in\mathcal{M}_{\leq 3X}}\frac{\Lambda_X(f)\chi_D(f)}{|f|^{\frac12+\sigma_0}}\right)+O\left(\frac{\delta}{X}+\frac{\lambda}{X(X+2)}\right)\\
&=\widetilde{P}_X(1/2+\sigma_0,\chi_D)+\frac{1}{2X^2}\sum_{X<d(f)\le 3X}\frac{\Lambda_X(f)\chi_D(f)}{d(f)|f|^{\frac12+\sigma_0}} \\
&+O\left(\frac{1}{X^3}\sum_{f\in\mathcal{M}_{\leq 3X}}\frac{\Lambda_X(f)\chi_D(f)}{|f|^{\frac12+\sigma_0}}\right)+O\left(\frac{\delta}{X}+\frac{\lambda}{X(X+2)}\right).
\end{align*}
\noindent
We conclude by noting that for $X<d(f)\le 3X$ we have $\frac{X}{d(f)}<1$, thus
\[\frac{1}{2X^2}\sum_{X<d(f)\le 3X}\frac{\Lambda_X(f)\chi_D(f)}{d(f)|f|^{\frac12+\sigma_0}}\le \frac{1}{2X^3}\sum_{f\in\mathcal{M}_{\leq 3X}}\frac{\Lambda_X(f)\chi_D(f)}{|f|^{\frac12+\sigma_0}}.\]
\end{proof}
	\section{Moment Estimation and Proof of the unconditional results}\label{thm1.1proof}
Recall that:  the parameter $\sigma_0=\frac{c}{X}$, with $0<c<\frac1{2\log q}$.

	
	\begin{proposition}\label{main proposition}
	Let $k$  be a natural number such that  $k\ll \left(\frac{\log n}{(\log \log n)^2}\right)^{\frac{1}{3}}$ and  $1\le X\le \frac{n}{2k}$. Then, as $D$ varies in $\mathcal{H}_n$, the distribution of $\widetilde{P}_{X}\left(\frac{1}{2}+\sigma_o, \chi_D\right)$ is approximately normal with mean $\frac12 \log n$ and variance $\log n$.
	\end{proposition}
	
	\noindent
	To establish  Proposition \ref{main proposition}, we need to calculate the moments of $\widetilde{P}_X\left(\frac{1}{2}+\sigma_o, \chi_D\right)-\frac{1}{2}\log X$.
		
	\vspace{2mm}
	\noindent
	We start with the moments of the following sum over irreducible polynomials:  
	\[
	P_X\left(\frac{1}{2}+\sigma_o, \chi_D\right)=\sum_{d(P)\leq X}\frac{\chi_D(P)}{|P|^{\frac{1}{2}+\sigma_o}}.
	\]
	
	\begin{lemma}\label{moment estimations over irreducibles}
Assume that  $1\leq X\le \frac{n}{2k}$. Uniformly for all  odd natural numbers  $k\leq {\frac{n}{2}}$ and for every $\epsilon>0$, 
	\begin{align*}
	\frac{1}{|\mathcal{H}_n|}\sum_{D\in \mathcal{H}_n}\left(P_X\left(\frac{1}{2}+\sigma_o, \chi_D\right)\right)^k\ll_{\epsilon} \frac{q^{\frac{n}{2}+\left(\frac{1}{2}+\epsilon-\sigma_o\right)Xk}}{X^k},
	\end{align*}
while, uniformly for all even natural numbers  $k\ll (\log n)^{\frac{1}{3}}$, 
	\begin{align*}
	\frac{1}{|\mathcal{H}_n|}\sum_{D\in \mathcal{H}_n}\left(P_X\left(\frac{1}{2}+\sigma_o, \chi_D\right) \right)^k=\frac{k!}{(k/2)! 2^{\frac{k}{2}}}(\log X)^{\frac{k}{2}}
\left(1+O\left(\frac{k^3}{\log X}\right)\right).
	\end{align*}
\end{lemma}
\begin{remark}
In the case of function fields, the contribution of the odd moments is negligibly small even for $k$ large due to the fact that Lindel{\"o}f Hypothesis is known. For example, if we take $X=\frac{n}{2k}$ then the odd moments contribute $o(1)$.
\end{remark}

\begin{proof}
Expanding $k$-th power and interchanging summations, we have
\[
s_{n}=\sum_{\substack{d(P_i)\leq X\\ i=1, \ldots, k}}\frac{1}{|P_1\ldots P_k|^{\frac{1}{2}+\sigma_o}}\sum_{D\in \mathcal{H}_{n}}\chi_D(P_1\ldots P_k).
\]
\begin{case}
Let $k$ be even and $P_1\ldots P_k=\square$.
\begin{subcase}
Suppose there are exactly $\frac{k}{2}$ distinct primes say $Q_1, \ldots, Q_{\frac{k}{2}}$ with
\[
d(Q_1)\leq d(Q_2)\leq \ldots \leq d(Q_{k/2}).
\]
Using Lemma \ref{square term evaluation}, we obtain 
\begin{align*}
&s_{n}= \frac{k!}{(k/2)! 2^{\frac{k}{2}}}\sum_{\substack{d(Q_i)\leq X\\ i=1, \ldots, k/2}}\frac{1}{|Q_1\ldots Q_{k/2}|^{1+2\sigma_o}}\sum_{D\in \mathcal{H}_{n}}\chi_D\left((Q_1\ldots Q_{k/2})^2\right)\\
&=|\mathcal{H}_{n}|\times \frac{k!}{(k/2)! 2^{\frac{k}{2}}}\sum_{\substack{d(Q_i)\leq X\\ i=1, \ldots, k/2}}\prod_{i=1}^{k/2}\frac{1}{(|Q_i|+1)|Q_i|^{2\sigma_o}}
+ O\bigg(|\mathcal{H}_n|^{-1}\frac{k!}{(k/2)! 2^{\frac{k}{2}}} \sum_{\substack{d(Q_i)\leq X\\ i=1, \ldots, k/2}}\frac{1}{|Q_1\ldots Q_{k/2}|}\bigg)\\
&=|\mathcal{H}_{n}|\times \frac{k!}{(k/2)! 2^{\frac{k}{2}}}\sum_{\substack{d(Q_i)\leq X\\ i=1, \ldots, k/2}}\prod_{i=1}^{k/2}\frac{1}{(|Q_i|+1)|Q_i|^{2\sigma_o}}+O\left(|\mathcal{H}_n|^{-1} \frac{k!}{(k/2)! 2^{\frac{k}{2}}}(\log X)^{\frac{k}{2}}\right).
\end{align*}
Now, if we fix $Q_1, \ldots, Q_{\frac{k}{2}-1}$ and vary over $Q_{k/2}$ in the main term of the above expression. Then, the sum becomes 
\[
\sum_{d(Q_{k/2})\leq X}\frac{1}{(|Q_{k/2}|+1)|Q_{k/2}|^{2\sigma_o}}+O\left(\frac{k}{2}\right).
\]

\vspace{2mm}
\noindent
Applying Lemma \ref{upper bound for truncated sum over primes }, we find that
\begin{align*}
s_{n}&=|\mathcal{H}_{n}|\times \frac{k!}{(k/2)! 2^{\frac{k}{2}}}\left(\log X+O(k)\right)^{\frac{k}{2}}\\
&=|\mathcal{H}_{n}|\times \frac{k!}{(k/2)! 2^{\frac{k}{2}}}(\log X)^{\frac{k}{2}}\left(1+O\left(\frac{k}{\log X}\right)\right)^{\frac{k}{2}}.
\end{align*}
\end{subcase}

\begin{subcase}
Let there are $Q_1, \ldots, Q_r$ distinct primes with 
\[
d(Q_1)\leq \ldots \leq d(Q_k)\]
 where $r<\frac{k}{2}$. Therefore, Lemma \ref{upper bound for truncated sum over primes } gives us  
\begin{align*}
s_n\leq & |\mathcal{H}_{n}|\times \sum_{r<\frac{k}{2}}\frac{k!}{r!2^r}\binom{k-r}{r}\bigg(\sum_{d(P)\leq X}\frac{1}{(|P|+1)|P|^{2\sigma_o}}\bigg)^r\\
&+O\left(|\mathcal{H}_n|^{-1}\sum_{r<\frac{k}{2}}\frac{k!}{r!2^r}\binom{k-r}{r}(\log X)^r\right)\\
& \leq |\mathcal{H}_{n}|\times \sum_{r<\frac{k}{2}}\frac{k!}{r!2^r}\binom{k-r}{r}\left(\log X+O(1)\right)^r+O\left(|\mathcal{H}_n|^{-1}\sum_{r<\frac{k}{2}}\frac{k!}{r!2^r}\binom{k-r}{r}(\log X)^r\right)\\
&\ll |\mathcal{H}_{n}|\times \frac{k^3 k!}{(k/2)!2^{k/2}\log X}(\log X)^{\frac{k}{2}}+O\left(|\mathcal{H}_n|^{-1}\frac{k^3 k!}{(k/2)!2^{k/2}\log X}(\log X)^{\frac{k}{2}}\right).
\end{align*}
\end{subcase}
\end{case}

\begin{case}\label{even and nonsquare case}
Let $k$ be even and $P_1\ldots P_k\neq \square$. 
Then we use Lemma \ref{polya-vinogradov inequality} to get
\[
s_{n}\ll \sqrt{|\mathcal{H}_n|} \sum_{\substack{d(P_i)\leq X\\ i=1, \ldots, k}}\frac{1}{|P_1\ldots P_k|^{\frac{1}{2}+\sigma_o-\epsilon}}\ll  \frac{q^{\frac{n}{2}+\left(\frac{1}{2}+\epsilon-\sigma_o \right) Xk}}{X^k}.
\]
\end{case}

\begin{case}
Let $k$ be odd. In this case  $P_1\ldots P_k\neq \square$.
In a similar manner in Case \ref{even and nonsquare case}, we come across the following 
\[
s_{n}\ll \frac{q^{\frac{n}{2}+\left(\frac{1}{2}+\epsilon-\sigma_o \right) Xk}}{X^k}.
\]
\end{case}

\end{proof}
	\begin{lemma}\label{moment estimation}
	Suppose that the hypothesis in Lemma \ref{moment estimations over irreducibles} holds. Uniformly for all  odd natural numbers  $k\ll \left(\frac{\log X}{(\log \log n)^2}\right)^{\frac{1}{3}}$, we have
	\begin{align*}
	\frac{1}{|\mathcal{H}_n|}\sum_{D\in \mathcal{H}_n}\left(\widetilde{P}_X\left(\frac{1}{2}+\sigma_o, \chi_D\right)- \frac{1}{2}\log X\right)^k\ll \frac{k!\, k^{\frac{3}{2}}(\log X)^{\frac{k}{2}}}{(k/2)!\, 2^{\frac{k}{2}}}\frac{(\log \log n)^{2}}{(\log X)^{\frac{1}{2}}},
	\end{align*}
	while, uniformly for all even natural numbers $k\ll \left(\frac{\log X}{(\log \log n)^2}\right)^{\frac{1}{3}}$,  we obtain
	\begin{align*}
	\frac{1}{|\mathcal{H}_n|}&\sum_{D\in \mathcal{H}_n}\left(\widetilde{P}_X\left(\frac{1}{2}+\sigma_o, \chi_D\right)- \frac{1}{2}\log X \right)^k\\
	&=\frac{k!}{(k/2)! 2^{\frac{k}{2}}}(\log X)^{\frac{k}{2}}
	\left(1+O\left(\frac{k^3}{\log X}\right)+O\left(\frac{k^{\frac{3}{2}}(\log \log n)^{2}}{(\log X)^{\frac{1}{2}}}\right)\right).
	\end{align*}
	\end{lemma}

\begin{proof}[Proof of Proposition \ref{main proposition}]
From Lemma \ref{moment estimation}, the moments of $\widetilde{P}_X\left(\frac{1}{2}+\sigma_o, \chi_D\right)-\frac{1}{2}\log X$  exactly match the moments of a Gaussian random variable with mean $0$ and variance $\log n$, and since the Gaussian is determined by its moments, our proposition follows.
\end{proof}
\noindent
\begin{proof}[Proof of Theorem \ref{Unconditional CLT near to half line}] 
We first use Proposition \ref{logL-as-PX} to write $\log L(\tfrac12+\sigma_0, \chi_D)$ in terms of $\widetilde{P}_X(\tfrac12+\sigma_0,\chi_D)$, we take real parts of the expression and then apply Proposition \ref{main proposition} with  $X=\frac{n}{4k}$. From here we have the estimate for the main term when $k \ll \left(\frac{(\log n)}{(\log\log n)^2}\right)^{\frac{1}{3}}$. Finally, we use Lemma  \ref{mean square estimate} to show the error term coming from Proposition \ref{logL-as-PX} is negligible.
\end{proof}

\noindent
It remains to prove Lemma \ref{moment estimation}.
	\begin{proof}[Proof of Lemma \ref{moment estimation}]

	We write
	\[
	\widetilde{P}_X\left(\frac{1}{2}+\sigma_o, \chi_D\right)=\sum_{d(P)\leq X}\frac{\Lambda(P)\chi_D(P)}{d(P)|P|^{\frac{1}{2}+\sigma_o}}+\sum_{d(P)\leq \frac{X}{2}}\frac{\Lambda(P^2)\chi_D(P^2)}{d(P^2)|P|^{1+2\sigma_o}}+O\left(\sum_{d(P)\leq \frac{X}{3}}\frac{1}{|P|^{\frac{3}{2}+3\sigma_o}}\right).
	\]
It is clear that the third sum is $O(1)$. So the contribution of $P^k$, $k\geq 3$ in $\widetilde{P}_{X}\left(\frac{1}{2}+\sigma_o, \chi_D\right)$ is $O(1)$.  Therefore, Lemma \ref{upper bound for truncated sum over primes } yields
\begin{align*}
\widetilde{P}_X\left(\frac{1}{2}+\sigma_0, \chi_D\right)=&\sum_{d(P)\leq X}\frac{\chi_D(P)}{|P|^{\frac{1}{2}+\sigma_o}}+\frac12\sum_{\substack{P\nmid D\\d(P)\leq \frac{X}{2}}}\frac{1}{|P|^{1+2\sigma_0}}+O(1)\\
& = \sum_{d(P)\leq X}\frac{\chi_D(P)}{|P|^{\frac{1}{2}+\sigma_o}}+\frac{1}{2}\log X+O(\log \log n),
\end{align*}
where the error term $O(\log\log n)$ comes from the sum over $P$ such that $P| D$.
Therefore, 
\[
\widetilde{P}_X\left(\frac{1}{2}+\sigma_0, \chi_D\right)-\frac12\log X=P_X\left(\frac{1}{2}+\sigma_o, \chi_D\right)+O(\log\log n).
\]
This implies that for some positive constant $c$,
\begin{align}\label{expanding k th power}
\sum_{D\in \mathcal{H}_n}\left(\widetilde{P}_X\left(\frac{1}{2}+\sigma_0, \chi_D\right)-\frac12\log X\right)^k=\sum_{D\in \mathcal{H}_n}\left(P_X\left(\frac{1}{2}+\sigma_o, \chi_D\right)\right)^k\\
 +O\left(\sum_{D\in \mathcal{H}_n}\sum_{r=1}^{k-1}\binom{k}{r}(c\log \log n)^{k-r}\bigg|P_X\left(\frac{1}{2}+\sigma_o, \chi_D\right)\bigg|^r\right).\nonumber
\end{align}

\vspace{2mm}
\noindent
The first sum will be handled by Lemma \ref{moment estimations over irreducibles}.  To handle reminder term we need to estimate the sum 
\[
\sum_{D\in \mathcal{H}_n}\bigg|P_X\left(\frac{1}{2}+\sigma_o, \chi_D\right)\bigg|^r,  \quad r\leq k-1.
\]
If $r$ is odd then Cauchy-Schwarz inequality and Lemma \ref{moment estimations over irreducibles} with even moments gives us
\begin{align*}
\sum_{D\in \mathcal{H}_n}\bigg|P_X\left(\frac{1}{2}+\sigma_o, \chi_D\right)\bigg|^r & \leq \bigg(\sum_{D\in \mathcal{H}_n}\bigg(P_X\bigg(\frac{1}{2}+\sigma_o, \chi_D\bigg)\bigg)^{r-1}\bigg)^{1/2}\\
&\times \bigg(\sum_{D\in \mathcal{H}_n}\bigg(P_X\bigg(\frac{1}{2}+\sigma_o, \chi_D\bigg)\bigg)^{r+1}\bigg)^{1/2}\\
& \ll |\mathcal{H}_n|\left(\frac{(r-1)!\, (r+1)!}{\left(\frac{r-1}{2}\right)! \left(\frac{r+1}{2}\right)! \, 2^r}\right)^{1/2} (\log X)^r\\
&=  |\mathcal{H}_n|\frac{(r-1)!\, r^{1/2}}{\left(\frac{r-1}{2}\right)! \, 2^{\frac{r-1}{2}}}(\log X)^r.
\end{align*}
On the other hand, if $r$ is even then directly use the asymptotic formula for even moments.
Finally, plugging these estimates with the Lemma \ref{moment estimations over irreducibles} to the sum \eqref{expanding k th power}, we complete the proof.

\end{proof}
\section{The difference between $L$-functions and Proof of Theorem \ref{main CLT theorem}}
%
%
%
Recall that:  the parameter $\delta$ as defined as in \eqref{delta}, the parameter $\sigma_0=\frac{c}{X}$, with $0<c<\frac1{2\log q}$. The following proposition gives the connection between the $L$-function at $\frac{1}{2}$ and very close to $\frac{1}{2}$.
We define 
\[\mathcal{H}_{n,0}=\left\{D\in \mathcal{H}_{n}: \mathcal{L}(q^{-\frac12}e(\theta_{j,D}),\chi_D)=0\Rightarrow \min_j|\theta_{j,D}|>\frac1{y\delta}\right\}.\]
\begin{proposition}\label{propforproofthm1.3}
Let $X\le \frac{n}{4k}$ but goes to infinity with $n$. Suppose that $\delta\, \sigma_o=o(\sqrt{\log \delta})$ as $\delta \to \infty$. Then, we have
\[
\frac1{|\mathcal{H}_n|}\sum_{D\in \mathcal{H}_{n,0}} \Big|\log \Big|L\left(\frac{1}{2}, \chi_D\right)\Big|-\log\Big|L\left(\frac{1}{2}+\sigma_o, \chi_D\right)\Big|\Big|=o\left(\sqrt{\log \delta}\right).
\]
\end{proposition}
\begin{remark}
Notice that there is a huge difference between the quantities $\log \Big|L\left(\frac{1}{2}, \chi_D\right)\Big|-\log\Big|L\left(\frac{1}{2}+\sigma_o, \chi_D\right)\Big|$ and $ \log \Big|L\left(\frac{1}{2}+\sigma_o, \chi_D\right)\Big|-\log\Big|L\left(\frac{1}{2}, \chi_D\right)\Big|$. The first quantity is very small whereas the second one is quite large and this is the reason for summing over $\mathcal{H}_{n,0}$ in the above Proposition. 
\end{remark}
\begin{proof}

Recall that 
\begin{align}\label{log derivative}
\frac{L'}{L}(s, \chi_D)=\log q\bigg( \frac{\lambda q^{-s}}{1-q^{-s}}-\delta+ \sum_{j=1}^{2\delta}\Re\bigg(\frac{1}{1-\alpha_jq^{\frac{1}{2}-s}}-\frac{1}{2}\bigg)\bigg).
\end{align}
Since $D\in\mathcal{H}_{n,0}$, we have $L(\frac{1}{2}, \chi_D)\neq 0$.
Then using the fact that \eqref{log derivative} is real for $s\in \mathbb{R}$, we obtain
\begin{align}\label{key-for-thm1.3}
\nonumber &\log \left|L\left(\frac{1}{2}+\sigma_o, \chi_D\right)\right|- \log \left|L\left(\frac{1}{2}, \chi_D\right)\right|=\int_{0}^{\sigma_o}\frac{L'}{L}\left(\frac{1}{2}+\sigma, \chi_D\right)d\sigma\\
\nonumber&= \int_{0}^{\sigma_o}\log q\bigg( \frac{\lambda q^{-\tfrac12-\sigma}}{1-q^{-\tfrac12-\sigma}}-\delta+\sum_{j=1}^{2\delta}\Re\bigg(\frac{1}{1-\alpha_jq^{-\sigma}}-\frac{1}{2}\bigg)\bigg)d\sigma\\
\nonumber& =\delta \sigma_o + \frac{\log q}{2}\sum_{j=1}^{2\delta}\Re \int_{0}^{\sigma_o}\frac{1+\alpha_j q^{-\sigma}}{1-\alpha_j q^{-\sigma}}d\sigma +\lambda\log q \log\left(\frac{1-q^{-\tfrac12-\sigma_0}}{1-q^{-\tfrac12}}\right)\\
&=\delta\sigma_o \log q -\frac{1}{2}\sum_{j=1}^{2\delta
}\log\left( \frac{q^{\sigma_o}+q^{-\sigma_o}-2\cos (2\pi \theta_{j, D})}{2-2\cos (2\pi\theta_{j, D})}\right)+\lambda\log q \log\left(\frac{1-q^{-\tfrac12-\sigma_0}}{1-q^{-\tfrac12}}\right).
\end{align}
 We observe that 
 \begin{align}\label{observation 1}
 \Re \bigg(\frac{1}{1-\alpha_jq^{-\sigma_o}}-\frac{1}{2}\bigg)=\frac{1-q^{-2\sigma_o}}{2\left(1- q^{-\sigma_o} \cos(2\pi \theta_{j, D})+ q^{-2\sigma_o}\right)}.
 \end{align}
 
 \vspace{2mm}
 \noindent
Now we use the inequality $\log(1+x)<x$, combined with the fact that $D\in \mathcal{H}_{n,0}$ and \eqref{observation 1} to get
\begin{align*}
\bigg|\log \frac{q^{\sigma_o}+q^{-\sigma_o}-2\cos (2\pi \theta_{j, D})}{2-2\cos (2\pi\theta_{j, D})}\bigg|& \leq \frac{q^{\sigma_o}+q^{-\sigma_o}-2}{2-2\cos (2\pi\theta_{j, D})} \ll \frac{\sigma_o^2}{2-2\cos (2\pi\theta_{j, D})}\\
&\ll \frac{y\sigma_o (1-q^{-2\sigma_o})}{2-2\cos (2\pi\theta_{j, D})}\ll \frac{y\sigma_o (1-q^{-2\sigma_o})}{1- q^{-\sigma_o} \cos(2\pi \theta_{j, D})+ q^{-2\sigma_o}}\\
&\ll y\sigma_o \, \Re \bigg(\frac{1}{1-\alpha_jq^{-\sigma_o}}-\frac{1}{2}\bigg).
\end{align*}

\vspace{2mm}
\noindent
On the real line, from \eqref{log derivative}, we see that 
\[
\frac{L'}{L}\left(\frac{1}{2}+ \sigma_o, \chi_D \right)=\log q\bigg(\frac{\lambda q^{-\tfrac12-\sigma_0}}{1-q^{-\tfrac12-\sigma_0}} -\delta+\sum_{j=1}^{2\delta}\Re \bigg(\frac{1}{1-\alpha_j q^{-\sigma_o}}-\frac{1}{2}\bigg)\bigg).
\]
From this expression of the logarithmic derivative of $L(s, \chi_D)$ on the real line, we have
\begin{align}\label{bound of log along zeros}
\bigg|\log \frac{q^{\sigma_o}+q^{-\sigma_o}-2\cos (2\pi \theta_{j, D})}{2-2\cos (2\pi\theta_{j, D})}\bigg|\ll y\sigma_o \, \frac{L'}{L}\left(\frac{1}{2}+\sigma_o, \chi_D\right) + y\delta \sigma_o.
\end{align}

\vspace{2mm}
\noindent
Using \eqref{LprimeL-identity}, we deduce that  
\begin{align}\label{bound for log derivative}
\frac{L'}L\left(\frac12+\sigma_0,\chi_D\right)=O\left(\frac1{X^2}\sum_{f \in \mathcal{M}_{\leq 3X}}\frac{\Lambda_X(f)\chi_D(f)}{|f|^{\frac12+\sigma_0}}\right)+O(\delta).
\end{align} 
After inserting the bound \eqref{bound for log derivative} into \eqref{bound of log along zeros}, we obtain 
 \[
 \log \left|L\left(\frac{1}{2}, \chi_D\right)\right|- \log\left| L\left(\frac{1}{2}+\sigma_o, \chi_D\right)\right|\ll \delta\sigma_o+ y\delta\sigma_o + \frac{y \sigma_o}{X^2} 
\sum_{f \in \mathcal{M}_{\leq 3X}}\frac{\Lambda_X(f)\chi_D(f)}{|f|^{\frac12+\sigma_0}}.
 \]
 Now we choose $y=y(\delta)$ such a way that 
 \[
 y\delta\sigma_o= o(\sqrt{\log_q \delta}) \quad \text{ and } \quad \delta\sigma_o= o(\sqrt{\log_q \delta}).
 \]
 Therefore, it is enough to prove
 \[
 \frac{1}{|\mathcal{H}_n|}\sum_{D\in \mathcal{H}_{n,0}}\bigg|
\frac{y \sigma_o}{X^2}\sum_{f \in \mathcal{M}_{\leq 3X}}\frac{\Lambda_X(f)\chi_D(f)}{|f|^{\frac12+\sigma_0}}\bigg|^2= O(y^2)=o(\log \delta),
 \]
We see that 
\begin{align*}
\sum_{d(f)\leq X}\frac{\Lambda_X(f)\chi_D(f)}{|f|^{1/2+\sigma_0}}\ll \sum_{d(f)\leq X}d(f)q^{(-1/2-\sigma_0)d(f)}\ll \sum_{n\leq X}nq^{\frac{n}{2}-n\sigma_0}\ll X^2q^{\frac{X}{2}}
\end{align*}
Since  $|\mathcal{H}_n\setminus \mathcal{H}_{n,0}|=o(1)$, Lemma \ref{mean square estimate} therefore yields
\begin{align*}
 \frac{1}{|\mathcal{H}_n|}\sum_{D\in \mathcal{H}_{n,0}}\bigg|
\frac{y \sigma_o}{X^2}\sum_{f \in \mathcal{M}_{\leq 3X}}\frac{\Lambda_X(f)\chi_D(f)}{|f|^{\frac12+\sigma_0}}\bigg|^2= o(y^2)+O\left(\frac{y^2q^{X-n}}{X^2}\right).
\end{align*}
The above choice of $X$ gives us $o(y^2)$.
 
\end{proof}
\begin{proof}[Proof of Theorem \ref{main CLT theorem}]
We will deduce this result by making use of Hypothesis \ref{low lying zero hypothesis} with the choices $\sigma_0=\frac{c}{X}$, $y=y(\delta)$ such that 
 \[
 y\delta\sigma_0= o(\sqrt{\log_q \delta}) \quad \text{ and } \quad \delta\sigma_o= o(\sqrt{\log_q \delta}).
 \]
Hypothesis \ref{low lying zero hypothesis} states that there exists a subset $\mathcal{H}_{n,0}$ of $\mathcal{H}_n$ with measure $1-o(1)$  as $n\to\infty$, such that if $D\in \mathcal{H}_{n,0}$  then any zero of $L(s,\chi_D)$ satisfies the property that \[\min_{j}|\theta_{j, D}|> \frac{1}{y\delta}.\] 
Restricting to $\mathcal{H}_{n,0}$, by Prpoposition \ref{propforproofthm1.3},
we conclude
\[
\frac1{|\mathcal{H}_n|}\sum_{D\in \mathcal{H}_{n,0}} \Big|\log \left|L\left(\frac{1}{2}, \chi_D\right)\right|-\log \left|L\left(\frac{1}{2}+\sigma_o, \chi_D\right)\right|\Big|=o\left(\sqrt{\log \delta}\right),
\]
thus this does not alter the distribution and we have from Theorem \ref{Unconditional CLT near to half line} that the distribution of $L(\frac12,\chi_D)$ is also approximately normal with mean $\frac12\log n$ and variance $\log n$ as desired. 
 
%
\end{proof}

\section{Proof of Corollary \ref{Unconditional CLT}}
\noindent
We use \eqref{log derivative} to obtain
\begin{align*}
&\log \Big|L\left(\frac{1}{2}, \chi_D\right)\Big|- \log \Big|L\left(\frac{1}{2}+\sigma_o, \chi_D\right)\Big|=\Re \int_{0}^{\sigma_o}- \frac{L'}{L}\left(\frac{1}{2}+\sigma, \chi_D\right)d\sigma\\
&=\Re \int_{0}^{\sigma_o}\log q\bigg(\frac{\lambda q^{-\tfrac12-\sigma}}{1-q^{-\tfrac12-\sigma}} -\delta+\sum_{j=1}^{2\delta}\bigg(\frac{1}{1-\alpha_jq^{-\sigma}}-\frac{1}{2}\bigg)\bigg)d\sigma\\
& =\delta \sigma_o \log q - \frac{\log q}{2}\sum_{j=1}^{2\delta}\Re \int_{0}^{\sigma_o}\frac{1+\alpha_j q^{-\sigma}}{1-\alpha_j q^{-\sigma}}d\sigma +\lambda\log q \log\left(\frac{1-q^{-\tfrac12-\sigma_0}}{1-q^{-\tfrac12}}\right)\\
&=\delta\sigma_o \log q -\frac{1}{2}\sum_{j=1}^{2\delta
}\log \frac{q^{\sigma_o}+q^{-\sigma_o}-2\cos (2\pi \theta_{j, D})}{2-2\cos (2\pi\theta_{j, D})}+\lambda\log q \log\left(\frac{1-q^{-\tfrac12-\sigma_0}}{1-q^{-\tfrac12}}\right)\\
&=\delta\sigma_o \log q -\frac{1}{2}\sum_{j=1}^{2\delta
}\log \bigg(1+\frac{q^{\sigma_o}+q^{-\sigma_o}-2}{2-2\cos (2\pi\theta_{j, D})}\bigg)+\lambda\log q \log\left(\frac{1-q^{-\tfrac12-\sigma_0}}{1-q^{-\tfrac12}}\right).
\end{align*}

Note that the summands in the second term are all non-negative which leads the second term of the above expression to be a negative value. Therefore, we deduce that
\[
\log \Big|L\left(\frac{1}{2}, \chi_D\right)\Big|\leq \log \Big|L\left(\frac{1}{2}+\sigma_o, \chi_D\right)\Big|+ \delta \sigma_o \log q+\lambda\log q \log\left(\frac{1-q^{-\tfrac12-\sigma_0}}{1-q^{-\tfrac12}}\right).
\]

\vspace{2mm}
\noindent
From the Hypothesis that $\frac{\delta \sigma_o}{\sqrt{\log \delta}}=o(1)$, we obtain
\[
\frac{\log \Big|L\left(\frac{1}{2}, \chi_D\right)\Big|}{\sqrt{\log \delta}}\leq \frac{\log \Big|L\left(\frac{1}{2}+\sigma_o, \chi_D\right)\Big|}{\sqrt{\log \delta}}+ o(1).
\]
Hence, the theorem follows from the Theorem \ref{Unconditional CLT near to half line}.


\end{document}